\documentclass[12pt]{amsart}
\usepackage{amscd, fullpage,amssymb}
\usepackage[all]{xy}
\usepackage{xcolor}
\newtheorem{thm}{Theorem}[section]
\newtheorem{prop}[thm]{Proposition}
 
\newtheorem{lem}[thm]{Lemma}
\newtheorem{cor}[thm]{Corollary}

\theoremstyle{remark}

\theoremstyle{definition}
\newtheorem{defn}[thm]{Definition}

\newcommand{\id}{{\rm id}}

\newcommand{\N}{\mathbb{N}}

\newcommand{\F}{\mathbb{ F}}

\newcommand{\Z}{\mathbb{ Z}}

\newcommand{\Tor}{\operatorname{Tor}}

\newcommand{\im}{{\rm im\,}}

\newcommand{\HH}{\operatorname{H}}

\newcommand{\BP}{\operatorname{BP}}

\newcommand{\Zp}{\mathbb{Z}/p}

\newcommand{\coker}{\operatorname{coker}}

\newcommand{\CP}{\mathbb{ C}{\rm P}}
\begin{document}
\title{Bordism of elementary abelian groups via inessential Brown-Peterson homology}

\begin{abstract} We compute the equivariant bordism of free oriented $(\Z/p)^n$-manifolds
as a module over $\Omega_*^{SO}$, when $p$ is an 
odd prime.   We show, among others,  that this module 
 is canonically isomorphic to a direct sum of suspensions of 
multiple tensor products of $\Omega^{SO}_*(B \Z/p)$, and that it is generated by products of standard 
lens spaces. This considerably improves previous calculations by various authors. 

 Our approach relies on the investigation of the  submodule of the Brown-Peterson homology
 of $B (\Z/p)^n$ generated by elements coming from proper subgroups of $(\Z/p)^n$. 

We apply our results to the Gromov-Lawson-Rosenberg conjecture for atoral 
manifolds whose fundamental groups are elementary 
abelian of odd order. 

\end{abstract} 
\author{Bernhard Hanke}
\address{Institut f\"ur Mathematik, Universit\"at Augsburg, D-86135 Augsburg, Germany}
\email{hanke@math.uni-augsburg.de}

\date{\today;  \copyright \, Bernhard Hanke 2016}

\maketitle
 
\section{Overview} \label{intro} 

Conner and Floyd, in their seminal work \cite{CF},
 introduced and studied bordism groups of free oriented $G$-manifolds for 
finite groups $G$. This is equivalent to the oriented bordism of $BG$, the classifying space of $G$. 
Of fundamental interest are the elementary abelian groups $G = (\Z/p)^n$, where $p$ is a prime. 

Recall that $\Omega^{SO}_*(B\Z/p)$, the oriented bordism of free $\Z/p$-manifolds, 
is generated as a module over $\Omega^{SO}_*= \Omega^{SO}_*(pt.)$ by elements
$z_{m} \in \Omega^{SO}_{2m+1}(B\Z/p)$, $ m \geq 0$,  represented by classifying maps 
$L^{2m+1} \to B \Z/p$ of standard lens spaces $L^{2m+1} = S^{2m+1} / ( \Z/p) $. These  correspond to 
spheres $S^{2m+1}$ equipped with standard free $\Z/p$-actions of weight $(1, \ldots, 1)$. 

For odd $p$ the oriented bordism of (classifying spaces of) elementary abelian $p$-groups  fits 
into Landweber's exact K\"unneth sequence \cite[Theorem A]{Landweber}

\begin{align} \label{Land} 
0 \to \Omega^{SO}_*(B( \Z/p)^{n-1}) \otimes_{\Omega^{SO}_*} \Omega^{SO}_*(B\Z/p) \to 
        \Omega^{SO}_* (B( \Z/p)^n) \to \\ \nonumber \to \left( \Tor_{\Omega^{SO}_*}(\Omega^{SO}_*(B(\Z/p)^{n-1}), \Omega^{SO}_*(B\Z/p))
        \right)_{*-1} \to 0.
\end{align} 

The calculation of the middle term by induction on $n$
requires a splitting of this sequence as $\Omega^{SO}_*$-modules. Indeed, individual elements of the torsion product 
can be lifted to $\Omega^{SO}_* (B( \Z/p)^n)$ by a matrix Toda bracket construction,  see \cite[p. 195]{BR1} and \cite{Al}, but this involves choices (of zero bordisms) and hence does not give
 an $\Omega_*^{SO}$-linear 
splitting. Landweber observed that  for $n =2$  the parity of degrees of  elements in  $\Omega^{SO}_* (B( \Z/p)^2)$
 induces a canonial splitting, see \cite[Theorem 7.1]{Landweber}. 
One referee pointed out that more generally 
Holzsager's stable splitting of $B\Z/p$ in \cite{Holz} can be used to split Landweber's K\"unneth sequence
 for $n \leq 2(p-1)$
by considering degrees modulo $2(p-1)$ of elements in  $\Omega^{SO}_* (B(\Z/p)^n)$,  rather than modulo $2$. 
This leads to short proofs of Theorems \ref{splitting} and \ref{main} below. 

Despite a number of related contributions 
  \cite{CF, JW,JW2,  Landweber, Mit,RW, WilBP} 
it remained unclear 
how to construct a splitting of Landweber's exact sequence for all $n$. The following  result 
settles this problem. 

\begin{thm} \label{Landweber_split}  Landweber's exact K\"unneth sequence (\ref{Land})  splits $\Omega^{SO}_*$-linearly. There is a preferred  splitting, not depending  on any choices. 
\end{thm} 

For a full description of the $\Omega^{SO}_*$-module $\Omega^{SO}_*(B(\Z/p)^n)$ see Theorem \ref{main} and the following remarks. 
Theorem \ref{Landweber_split} will be proven by a  combination of algebraic and homotopy theoretic arguments, which do 
not give a geometric interpretation of the claimed  splitting, a priori. 
However, our proof of Theorem \ref{Landweber_split} still enables us  to identify a set of generators of $\Omega^{SO}_* (B(\Z/p)^n)$. For $n = 2$ 
Botvinnik-Gilkey \cite{BG} proved that the bordism classes that are represented by classifying maps 
  $L^{2m_1+1} \times L^{2m_2+1} \to B(\Z/p)^2$ 
or   by compositions $L^{2m+1} \to B\Z/p \stackrel{B\phi}{\longrightarrow} B (\Z/p)^2$ for the 
various 
group homomorphisms $\phi:\Z/p \to (\Z/p)^2$ span $\Omega^{SO}_* (B(\Z/p)^2)$
as an $\Omega^{SO}_*$-module. We will show that  this scheme generalizes to elementary abelian $p$-groups of arbitrary rank. 

Let $1 \leq k \leq n$ and let 
$\phi: (\Z/p)^k\to (\Z/p)^n$ be a group homomorphism, inducing a map of $\Omega^{SO}_*$-modules 
\[
     \phi_* : \Omega^{SO}_*(B(\Z/p)^k) \to \Omega^{SO}_*(B(\Z/p)^n) . 
\]
 If we represent an 
element $z \in \Omega^{SO}_{d} (B(\Z/p)^k)$ by a free oriented
 $(\Z/p)^k$-manifold 
 $M^d$, then $\phi_*(z)$ is represented by the free oriented $(\Z/p)^n$-manifold $(\Z/p)^n \times_{(\Z/p)^k} M$, where $(\Z/p)^k$ acts on $(\Z/p)^n$ by the map $\phi$. In the special case when 
$M$ can be taken as a product of $k$ odd dimensional spheres 
 with a product of standard free $\Z/p$-actions, we call  the quotient 
$\left( (\Z/p)^n \times_{(\Z/p)^k} M\right) / ( \Z/p)^n$ 
a  {\em generalized product of lens spaces}. 

\begin{thm}  \label{generated} Let $p$ be an odd prime. Then the $\Omega^{SO}_*$-module $\Omega_*(B ( \Z/p)^n)$ is generated by generalized products of lens spaces (including the empty product, represented by a point). 
 \end{thm} 
 
Note that the corresponding assertion for singular homology $\HH_*(B(\Z/p)^n ; \Z)$ is true only for $n =1$. 

The work on this paper started with an attempt to understand the proof of \cite[Theorem 5.6]{BR1}, which is a weak 
version of our Theorem \ref{generated} and represents the crucial step in the proof  
 of the Gromov-Lawson-Rosenberg conjecture about the 
existence of Riemannian metrics of 
positive scalar curvature on closed manifolds with elementary 
abelian fundamental groups given in \cite{BR1} and \cite{BR2}. It turned out that the proof 
of \cite[Theorem 5.6]{BR1} in {\em loc.~cit.~}is incorrect, but the theorem itself is true 
and can be proven with our methods. Theorem  \ref{generated} 
and its analogue for spin bordism lead to a  proof of  the Gromov-Lawson-Rosenberg conjecture 
for elementary abelian groups of odd order, see Section \ref{pscm}, where we also make some remarks on the argument in \cite{BR1}. 

{\em Acknowledgements.} This paper 
was written while visiting IMPA, Rio de Janeiro, whose hospitality is gratefully acknowledged.
We are grateful to Peter Landweber and to the referees for carefully reading  the first version of 
this manuscript and 
providing a number of valuable remarks. 
The research leading to this paper was supported by DFG grant HA 3160/6-1. 


\section{Brown-Peterson homology of elementary abelian groups} \label{calculation}

In this section we summarize our computations in some detail
and motivate our approach. For simplicity and in accordance with an analogous 
 convention in group homology we often suppress the letter $B$ in classifying spaces in our 
 computations of generalized homology groups of classifying spaces.

The classifying space $B(\Z/p)^n$ being $p$-local  we can restrict ourselves to a computation of  
$\Omega^{SO}_*((\Z/p)^n) = \Omega^{SO}_*(B(\Z/p)^n) $ localized at $p$.  We recall  \cite{WilBP} that for odd primes $p$, the $p$-local 
oriented bordism spectrum ${\rm MSO}_{(p)}$ splits into suspensions of copies of $\BP$, Brown-Peterson theory for the prime $p$, with coefficients
\[
    \BP_* = \Z_{(p)}[v_1, v_2, \ldots ], \; \deg(v_m) = 2p^m - 2. 
\]
Hence, for odd primes $p$, the computation of $\Omega_*^{SO}((\Z/p)^n)$ reduces to a computation 
of $\BP_*((\Z/p)^n)$. This is the theme of the paper at hand.

Recall that $B\Z/p$ has a CW structure with one cell in each dimension greater than or equal to 
zero \cite[Example (1.1.2)]{AP}. The associated reduced cellular chain complex $C_*$ has 
 one generator $c_d \in C_d$ for each
$d\geq 1$ and is equipped with the differential 
\[
     c_{2m} \mapsto p \cdot c_{2m-1} \, , \, \, c_{2m+1} \mapsto 0. 
\]
In particular, $\widetilde{\HH}_*(\Z/p)$, the reduced integral homology of $B \Z/p$, is an $\F_p$-module 
with generators $[c_{2m+1}] \in \widetilde{\HH}_{2m+1}(\Z/p)$, $m \geq 0$. 

The reduced homology $\widetilde{\HH}_*(\wedge^n \Z/p)$ of the $n$-fold smash product of $B\Z/p$ 
is  equal to the homology of the $n$-fold tensor product of $C_*$,
\[
    \widetilde{\HH}_*(\wedge^n \Z/p) = \HH_*(C_* \otimes_{\Z} \cdots \otimes_{\Z} C_*). 
\]
It follows from this description that the iterated  K\"unneth map 
\[
    \widetilde{\HH}_*(\Z/p) \otimes  \cdots \otimes  \widetilde{\HH}_*(\Z/p) \to \widetilde{\HH}_*(\wedge^n \Z/p)
\]
is injective and has a canonical splitting 
\[
  \Psi_n :   \widetilde{\HH}_*(\wedge^n \Z/p) \to \widetilde{\HH}_*(\Z/p) \otimes  \cdots \otimes
   \widetilde{\HH}_*(\Z/p)
\]
induced by the chain map
\[
   C_* \otimes \cdots \otimes C_* \to  \widetilde{\HH}_*(\Z/p) \otimes  \cdots \otimes
   \widetilde{\HH}_*(\Z/p)
\]
that sends each tensor product $c_{2m_1+1} \otimes \cdots \otimes c_{2m_n+1}$ to 
$[c_{2m_1+1}] \otimes \cdots \otimes [c_{2m_n+1}]$, and each tensor product involving an even dimensional 
generator $c_{2m}$ to zero.   

Theorem \ref{summary} of our paper  says  that  the reduced Brown-Peterson homology 
 $\widetilde{\BP}_*(\wedge^n \Z/p)$ has a similar chain model description, if $p$ is odd.  
 The next theorem   is our first  result in this direction. 
 
  From now 
on we assume that $p$ is an odd prime, if not stated otherwise.

\begin{thm} \label{splitting} There is a $\BP_*$-linear map 
\[
   \Psi_n : \widetilde{\BP}_*(\wedge^n \Z/p) \to \widetilde{\BP}_*(\Z/p) \otimes_{\BP_*} 
   \cdots \otimes_{\BP_*} \widetilde{\BP}_*( \Z/p) 
\]
which splits  the iterated K\"unneth map 
\[
    \Phi_n :   \widetilde{\BP}_*(\Z/p) \otimes_{\BP_*}
   \cdots \otimes_{\BP_*} \widetilde{\BP}_*(\Z/p)     \to \widetilde{\BP}_*(\wedge^n \Z/p) .
\]
The map $\Psi_n$ is a map of $\BP_* \BP$-comodules. 
\end{thm} 

Our construction of $\Psi_n$ is canonical and independent of any choices. The main idea 
is to study the {\em inessential Brown-Peterson group homology},  
the $\BP_*$-submodule of $\BP_*((\Z/p)^n)$ 
generated by elements coming from proper subgroups of $(\Z/p)^n$, see Definition \ref{inessential}. 
Our paper is based 
on the observation that this submodule is a complement of the image of the iterated K\"unneth map. 

Based on Theorem \ref{splitting} we can state our computation of  $\widetilde{\BP}_*(\wedge^n 
 \Z/p)$. By regarding $S^{\infty}$ both as a free contractible $\Z/p$- and $S^1$-space  we have a 
 canonical map 
  \[
 \pi :  B\Z/p \to \CP^{\infty} 
\]
induced by the standard inclusion $\Z/p \hookrightarrow S^1$. 
From this we get a map 
\[
  \gamma_{(X_i)}: \bigwedge_n B\Z/p \to \bigwedge_n  X_i 
\]
for each family $(X_i)_{1 \leq i \leq n}$, where each $X_i$ is equal to either $B\Z/p$ or $\CP^{\infty}$,
and we apply either $\pi$ or the identity to each smash product factor $B \Z/p$.

Recall that $\widetilde{\BP}_*(\CP^{\infty})$ 
 is a free $\BP$-module generated by elements $\beta_{m} \in \BP_{2m}(\CP^{\infty})$, $m \geq 1$,
 represented by the standard inclusion $ \CP^{m} \hookrightarrow \CP^{\infty}$. 
In particular, for any pointed space $X$, there is a canonical isomorphism 
\[
    \widetilde{\BP}_*(X \wedge \CP^{\infty}) \cong \widetilde{\BP}_*(X) \otimes_{\BP_*} 
   \widetilde{\BP}_*(\CP^{\infty}) .
\]

Hence, if  $k$ is the number of factors $X_i = B\Z/p$ in a smash product as before,
we can compose the induced map in reduced $\BP$-theory (involving a permutation of factors)
\[
     \widetilde{\BP}_*(\bigwedge_n \Z/p) \to \widetilde{\BP}_*(  \bigwedge_n X_i ) 
     \cong \widetilde{\BP}_*(\bigwedge_k \Z/p) \otimes_{\BP_*} \bigotimes_{n-k} \widetilde{\BP}_*(\CP^{\infty}) 
\]
 with the splitting $\Psi_k$ from Theorem \ref{splitting}, to get a $\BP_*$-linear map
\[
   \widetilde{\BP}_*(\wedge^n \Z/p) \to \bigotimes_{i=1}^n \widetilde{\BP}_*(X_i).
 \]
This is a map of $\BP_* \BP$-comodules.

For $k \geq 1$ let $(L_k)_*$ be the free graded $\BP_*$-module with generators $y_{m}$ in degree
$2m$, $0 <  m < p^k$.
We have a canonical $\BP_*$-linear projection

\begin{eqnarray*}
    \widetilde{\BP}_* ( \CP^{\infty}) & \to & (L_k)_* \\
                                   \beta_m & \mapsto & \begin{cases} y_{m}\,  {\rm~for~} 0 < m < p^k, \\ 
                                     0 {\rm~for~} m \geq p^k.  \end{cases}
\end{eqnarray*}
However, $(L_k)_*$ does not carry an induced $\BP_* \BP$-comodule structure.

 \begin{thm} \label{main}  These maps introduced so far induce an isomorphism of $\BP_*$-modules
 \[
   \Gamma_n:     \widetilde{\BP}_*(\wedge^n \Z/p)  \cong \bigoplus J_1 \otimes_{\BP_*}  \cdots \otimes_{\BP_*} J_n .
 \]
The direct sum is over all tensor products with $J_i$ equal to $\widetilde{\BP}_*(\Z/p)$ or to $(L_k)_*$, where $k$ is the number of $J_j$, $j<i$, with $J_j = \widetilde{\BP}_*(\Z/p) $. 
\end{thm} 

In \cite[Theorem 5.1]{JW} it was shown (for any prime) 
that $\widetilde{\BP}_*(\wedge^n \Z/p)$ has a $\BP_*$-module filtration whose associated 
graded module is $\BP_*$-isomorphic to the right hand side of  Theorem \ref{main}.
Our result  shows that the filtration can be omitted (for odd primes). Theorem \ref{main} implies a similar description 
of $\widetilde{\Omega}_*^{SO}(\wedge^n \Z/p ) = \widetilde{\BP}_*(\wedge^n \Z/p) \otimes_{\BP_*} (\Omega^{SO}_*)_{(p)}$. 
.

\begin{cor} \label{kun} The Landweber exact K\"unneth sequence \cite{Landweber} 
 \begin{align*}
    0 \to \widetilde{\BP}_*(\wedge^{n-1} \Z/p) \otimes_{\BP_*} \widetilde{\BP}_*(\Z/p) \to 
        \widetilde{\BP}_* (\wedge^n \Z/p) \to \\ \to \big( \Tor_{\BP_*}(\widetilde{\BP}_*(\wedge^{n-1} \Z/p), \widetilde{\BP}_*(\Z/p))\big)_{*-1} \to 0
\end{align*} 
splits $\BP_*$-linearly. 
\end{cor} 

In fact, with respect to Theorem \ref{main},  the tensor product on the left corresponds to the tensor product summands
with $J_n = \widetilde{\BP}_*(\Z/p)$,    and the torsion product on the right corresponds to the tensor product summands  with $J_n$ equal to some $(L_k)_*$. Corollary \ref{kun} immediately implies Theorem \ref{Landweber_split}.

We can conveniently summarize Theorems \ref{splitting} and \ref{main} 
 in terms of a chain model description of $\widetilde{\BP}_*(\wedge^n 
\Z/p)$ similar to the one for ordinary homology. 
Recall that $\widetilde{\BP}_*(\Z/p)$ is generated by elements
$z_{m} \in \widetilde{\BP}_{2m+1}(\Z/p)$, $ m \geq 0$,  represented by classifying maps 
$L^{2m+1} \to B \Z/p$ of standard lens spaces $L^{2m+1} = S^{2m+1} / ( \Z/p) $. These 
generators are subject to the relations
 \[
    \sum_{i=0}^{m}  a_{i} \cdot z_{m-i} = 0 
\]
where $a_{i} \in \BP_{2i}$ appear in the  formal group law for $\BP$-theory, cf. Section \ref{back}. 

Motivated by this calculation let  $C^{BP}_*$ be the free $\BP_*$-chain complex  in  one generator $c_{d}$ in each 
degree $d\geq 1$ and equipped with the $\BP_*$-linear differential $C^{BP}_* \to C^{BP}_{*-1}$ 
equal to 
\[
    c_{2m} \mapsto \sum_{i=0}^{m-1} a_{i} \cdot c_{2(m-i)-1} \, , \, \, \, c_{2m+1} \mapsto 0  .
\]
It is then clear that $\HH_*(C^{BP}_*) = \widetilde{\BP}_* ( \Z/p)$. Note the formal similarity to the chain complex 
$C_*$, computing the reduced singular homology $\widetilde{\HH}_*(\Z/p)$ considered before.

\begin{thm} \label{summary} There is a canonical $\BP_*$-linear  isomorphism 
\[
    \widetilde{\BP}_*(\wedge^n \Z/p) \cong \HH_* ( C^{BP}_* \otimes_{\BP_*} \cdots \otimes_{\BP_*} C^{BP}_*). 
\]
\end{thm} 

 Again, Holzsager's stable splitting of $B\Z/p$ in \cite{Holz} leads to an alternative  proof of 
 Theorem \ref{summary} for $n \leq 2(p-1)$.

The above isomorphism is compatible with the isomorphism of Theorem \ref{main} in the sense that the right hand 
side  maps isomorphically to the right hand side of Theorem \ref{main} 
by sending each generator 
$c_{2m+1}$ to $z_{m} \in \widetilde{\BP}_*(B\Z/p)$ 
and each generator $c_{2m}$ to $y_{m} \in (L_k)_*$ (respectively to $0$ for $m \geq p^k$).

The $\BP_*$-module on the right of Theorem \ref{main} is not invariant under 
the canonical ${\rm Sym}_n$-action permuting the factors of $(\Z/p)^n$, and does not carry an 
induced $\BP_* \BP$-comodule 
structure. However, we can equally well consider the induced map 
\[
   \widetilde{BP}_*(\wedge^n \Z/p) \to \bigoplus J_1 \otimes_{\BP_*}  \cdots \otimes_{\BP_*} J_n,
\]
where now the sum is  over all tensor products with $J_i = \widetilde{\BP}_*( \Z/p)$ or $J_i = 
\widetilde{\BP}_*(\CP^{\infty})$ and at least one occurence of $\widetilde{\BP}_*( \Z/p)$. As a corollary to 
Theorem \ref{main} this is an injective $\BP_*$-linear map and a map of 
$\BP_* \BP$-comodules. It is equivariant with respect to the natural ${\rm Sym}_n$-actions 
on both sides. 

Since the beginning of equivariant bordism theory \cite{CF}  the toral element $z_0 \otimes \cdots \otimes z_0 \in \BP_n(\wedge^n \Z/p)$ has played a significant role. The solution of the Conner-Floyed conjecture, 
stating that the annihilator ideal of this element is equal to $(p,v_1,  \ldots, v_{n-1}) \subset \BP_*$, took more than 
fifteen years and was finally achieved in  \cite{RW, Mit}. This shows that 
extraordinary group homology is difficult to calculate, even for groups as basic as $(\Z/p)^n$.

For us it is important that Ravenel-Wilson's solution of the Conner-Floyd conjecture not only gives information 
on the annihilator ideal of the toral class or injectivity of the iterated K\"unneth map \cite{JW}, but implies a stronger statement: The image of the iterated K\"unneth 
map on the one hand, and the $\BP_*$-submodule of $\BP_*( (\Z/p)^n)$ generated by elements coming from 
proper subgroups of $(\Z/p)^n$, the inessential Brown-Peterson group homology, on the other, intersect trivially. 

The largest part of our paper is devoted to showing that the image of the iterated K\"unneth 
map and the inessential Brown-Peterson homology span the whole of  $\BP_*((\Z/p)^n)$. 
This involves a detailed examination of the Pontryagin product on $\BP( (\Z/p)^n)$ induced by 
the group structure on $(\Z/p)^n$.
 
As a consequence of these calculations we get an interesting level structure on $\widetilde{\BP}_*(\wedge^n \Z/p )$, which 
nicely complements Theorem \ref{summary}. 

\begin{defn} \label{genprod} For $1 \leq k \leq n$ let 
\[
     \widetilde{\BP}_*^{(k)}(\wedge^n  \Z/p  ) \subset \BP_* (\wedge^n \Z/p) 
\]
be the $\BP_*$-submodule generated by the images of the compositions
\[
     \widetilde{\BP}_*(\Z/p) \otimes_{\BP_*}
   \cdots \otimes_{\BP_*} \widetilde{\BP}_*(\Z/p) \stackrel{\Phi_k}{\longrightarrow} 
      \BP_*((\Z/p)^k) \stackrel{\phi_*}{\longrightarrow} 
       \BP_*((\Z/p)^n) \to \widetilde{\BP}_*(\wedge^n \Z/p) 
\]
where on the left we take a $k$-fold tensor product and $\phi : (\Z/p)^k \to (\Z/p)^n$ is some group homomorphism.    
 
 \end{defn}

 \begin{thm} \label{level} There is a direct sum decomposition 
 \[
     \widetilde{\BP}_*(\wedge^n \Z/p) = \bigoplus_{k=1}^n \widetilde{\BP}_*^{(k)}(\wedge^n \Z/p)
 \]
 as $\BP_*$-modules and $\BP_* \BP$-comodules. The summand $\widetilde{BP}_*^{(k)}(\wedge^n \Z/p)$ 
 corresponds to the homology classes in Theorem \ref{summary} which are represented 
 by chains involving exactly $k$ odd degree generators $c_{2m+1}$. 
 \end{thm}

 Theorem \ref{level} implies Theorem \ref{generated} from the beginning of this paper. We remark that the 
 decomposition in Theorem \ref{level} is (already for $n = 2$) not induced by a stable splitting of $B(\Z/p)^n$, as 
 constructed
\cite{MP, Mit_Split} for instance, because the $\F_p$-homology of each of the resulting wedge summands 
must be invariant under the Bockstein homomorphism.

Some of our results, including the computations in Section \ref{compute}, carry over to the case $p=2$ and 
hence have implications for free stably almost complex $(\Z/2)^n$-manifolds. 
However, we will show at the end of section \ref{bewmagic} 
 that the image of the 
iterated K\"unneth map and the inessential Brown-Peterson homology do intersect nontrivially for 
$p=2$, so that not all of our methods cover the case $p=2$. In particular  Theorem \ref{level}
requires odd $p$. 
 This is remarkable, because the Conner-Floyd conjecture eventually holds for all primes, see  \cite{Mit} 
 and \cite[Appendix]{JW}. 
 We conjecture that  Theorem \ref{main} remains true for $p=2$.


\section{Outline of proof} \label{back} 

For more detailed information on the following material we refer to \cite{JW} 
and \cite{WilBP}.
We have $\BP^*(\CP^{\infty}) = \BP^*[[x]]$ 
with a generator  $x \in \BP^2(\CP^{\infty})$ equal to the first Conner-Floyd characteristic class. 
The group homomorphism $S^1 \to S^1$, $t \mapsto t^p$,  induces a map $p : BS^1  \to BS^1$ and hence, 
via the formal group law for Brown-Peterson theory and the model $BS^1 = \CP^{\infty}$, leads to an equation 
\[
     (p)^* x = \sum_{i=0}^{\infty}  a_{i} \cdot x^{1+i} \in \BP^*(\CP^{\infty})
\]
with $a_{i} \in  \BP_{2i}$, $i \geq 0$. It is well known that the generators $v_m \in \BP_{2p^m-2}$ can be chosen 
such that 
\[
     a_{p^m-1} \equiv v_m \mod (v_0, \ldots, v_{m-1}) 
\]
for $m \geq 0$. Here we set $v_0 := p$ so that in particular $a_0 = p$. 
We have elements $\beta_{m} \in \BP_{2m}(\CP^{\infty})$
dual to $x^m$. These are represented by the standard inclusions $\CP^{m} \hookrightarrow \CP^{\infty}$,
and satisfy the equation 
\[
   \beta_{m} \cap x  = \beta_{m-1} 
\]
for $m \geq 1$, where we set $\beta_0 := 1 \in \BP_0(\CP^{\infty})$. 
 The Gysin sequence associated to the fibration 
\[
     S^1 \hookrightarrow B \Z/p \stackrel{\pi}{\to} \CP^{\infty} 
\]
shows that 
\[
   \BP_{2m+1}(\Z/p) = \coker ( - \cap (p)^*x : \BP_{2m+2}(\CP^{\infty}) \to  \BP_{2m}(\CP^{\infty}) ).
\]
Hence the generators $\beta_{2m} \in \BP_{m}(\CP^{\infty})$ induce generators $z_{m} \in \BP_{2m+1}(\Z/p)$ 
for $m \geq 0$, that are subject to the relations 
\[
    \sum_{i=0}^{m} a_{i} \cdot z_{m-i} = 0. 
\]
Note that in \cite{JW} a slightly different notation is used, where $z_m$ denotes the generator in 
$\BP_{2m-1} ( \Z/p)$. We have $z_{m} = t_*(\beta_{m})$ with a stable transfer $t : \Sigma \CP^{\infty} \to B \Z/p$, see \cite[(2.12.)]{JW}. From this we obtain relations 
\[
   z_{m} \cap \pi^{*}(x) = z_{m-1}
\]
for $m\geq 1$. Also the stable transfer map allows us to compute the $\BP_* \BP$-comodule structure on 
$\widetilde{\BP}_*(\Z/p)$ from the one on $\widetilde{\BP}_*(\CP^{\infty})$, whis is well known, 
see for example \cite[Theorem 1.48]{WilBP}. 
In accordance with Section \ref{calculation} the element $z_{m}$ is represented by the classifying map 
$ L^{2m+1} \to B\Z/p$ of the standard lens space $L^{2m+1}$.

Following \cite{JW} we set
\[
   N_* := \widetilde{\BP}_*(\Z/p) 
\]
and use the shorthand 
\[
    N_*^k := N_* \otimes_{BP_*} \cdots \otimes_{BP_*} N_* 
\]
with $k$ tensor factors. 

From the above description of $\BP_{2m+1}(\Z/p)$ we get a free resolution 
\[
   (F_1)_* \stackrel{f_1}{\longrightarrow} (F_0)_* \stackrel{f_0}{\longrightarrow} N_{*-1} \to 0 
\]
where $F_1$ and $F_0$ are free graded $\BP_*$-modules in generators $y_m$ of degree $2m$, $m > 0$,
and 
\[
    f_1(y_{m}) = \sum a_{i} \cdot y_{m-i} \, , \, f_0(y_{m}) = z_{m-1}. 
\]
Hence for any graded $\BP_*$-module $M_*$ we have 
\[
   ( \Tor_{\BP_*}(M_*, N_*))_{m-1} = \ker \big( \id \otimes_{\BP_*} f_1 : M_* \otimes_{\BP_*} F_1 \to M_*
    \otimes_{\BP_*} F_0 \big)_m \subset (M_* \otimes_{\BP_*} F_1)_m. 
\]

The $\BP_*$-homology of $\wedge^n B \Z/p$ fits into Landweber's exact K\"unneth 
sequence \cite{Landweber} 
\[
  0 \to \widetilde{\BP}_*(\wedge^{n-1} \Z/p) \otimes_{\BP_*} N_* \to \widetilde{\BP}_* ( 
    \wedge^n \Z/p) \to \big(\Tor_{\BP_*}(\widetilde{\BP}_*(\wedge^{n-1} \Z/p), N_*)\big)_{*-1} \to 0.
\]
Using the above free resolution of $N_*$ the map $\widetilde{\BP}_* (\wedge^n \Z/p) \to \big(\Tor_{\BP_*}(\widetilde{\BP}_*(\wedge^{n-1} \Z/p), N_*)\big)_{*-1}$
has the following geometric interpretation, cf. \cite[Section 5]{JW}. We define $C$ as the stable cofibre 
in 
\[
       B\Z/p \stackrel{\pi}{\to} \CP^{\infty}  \to C. 
\]
After taking the smash product with $\wedge^{n-1} B\Z/p$ on the left, the induced long exact sequence in 
$\BP_*$-theory induces a short exact sequence
\[
     0 \to \coker ( \id \otimes_{\BP_*} f_1 ) \to \widetilde{BP}_*(\wedge^n \Z/p) \to \ker ( \id \otimes_{\BP_*} f_1) \to 0  
\]
which can be identified with the exact K\"unneth sequence (with a degree shift in the torsion 
product). In this description the map in 
the K\"unneth sequence is induced on the space level by 
\[
 \gamma_{(B\Z/p, \ldots, B\Z/p, \CP^{\infty})} =  \id \wedge \pi :  \bigwedge_n B \Z/p \to  \bigwedge_{n-1} B\Z/p \wedge \CP^{\infty}. 
\]

The proofs of Theorems \ref{splitting} and \ref{main} are parallel and by induction on $n$. We hence 
assume that Theorem \ref{main} holds for $n-1$.
According to Corollary \ref{wonderful} below (also see \cite[Theorem 4.1]{JW}) the composition 
\[
  ( \Tor_{\BP_*}(N_*^k , N_*))_{*-1} = \ker \big( \id \otimes_{\BP_*} f_1 :N_*^k \otimes_{\BP_*} F_1 \to N_*^k \otimes_{\BP_*}  F_0\big)_* \to N_*^k \otimes_{\BP_*} L_k 
\]
is an isomorphism for $1 \leq k \leq n-1$. We hence get a commutative diagram 
\[
   \xymatrix{
  \big( \Tor_{\BP_*}( \widetilde{\BP}_*( \wedge^{n-1}  \Z/p) , N_*) \big)_{*-1}    \ar@^{(->}[r]    \ar[d]_{\cong}^{\Tor_{\BP_*}(\Gamma_{n-1}, \id)}& \widetilde{\BP}_*(\wedge^{n-1}  \Z/p) \otimes_{\BP_*} \widetilde{\BP}_*(\CP^{\infty}) \ar[d]^{\Gamma_{n-1} \otimes {\rm proj.}}  \\ 
   \big( \Tor_{\BP_*} ( \bigoplus J_1 \otimes \cdots \otimes J_{n-1}, N_*)\big) _{*-1}\ar[r]^-{\cong} &  \bigoplus_{J_n = L_k} J_1 \otimes \cdots \otimes J_n                    
 } 
\]
where $\Gamma_{n-1}$ is taken from Theorem \ref{main} and the subscript $J_n = L_k$ indicates
that we are just considering those summands from Theorem \ref{main} with $J_n = (L_k)_*$ (with appropriate $k$).


 As in Theorem \ref{main} 
we have a canonical map 
\[
 \Gamma_n :   \widetilde{\BP}_*(\wedge^n \Z/p) \to \bigoplus_{J_1 \otimes \cdots \otimes J_n \neq N_*^n}  J_1 \otimes \cdots \otimes J_n,
\]
where the subscript means that we just sum over those tensor factors in Theorem \ref{main} 
that are not equal to $N_*^n$. This is achieved by first applying $\pi : B \Z/p \to \CP^{\infty}$ to at least one factor of $\wedge^n B\Z/p$ 
and then applying the splitting of Theorem \ref{main} to the $\BP$-homology of the 
smash product of the remaining copies of  $B\Z/p$. 
 Note that Theorem \ref{main} for $n-1$ (or less) is sufficient for this construction.

\begin{prop} \label{zerl}  Let 
\[
     K_*  \subset \widetilde{\BP}_*(\wedge^n \Z/p) 
\]
be a $\BP_*$-submodule and $\BP_* \BP$-sub-comodule which maps surjectively 
onto the image of the map 
\[
   \Gamma_n : \widetilde{\BP}_*(\wedge^n \Z/p) \to \bigoplus_{J_1 \otimes \cdots \otimes J_n \neq N_*^n}  J_1 \otimes \cdots \otimes J_n 
\]
defined above.

Let  $\Phi_n :   N_*^n    \to \widetilde{\BP}_*(\wedge^n \Z/p)$
be the iterated K\"unneth map and assume that 
\[
   K_* \cap \im \Phi_n   = 0 .
\]
Then both Theorems \ref{splitting} and \ref{main} hold for $n$.       

\end{prop}

\begin{proof} Using the K\"unneth sequence and the above commutative diagram the assumption on $K_*$ leads to 
\begin{equation} \label{summe} 
   \widetilde{\BP}_*( \wedge^n \Z/p) = K_* + \widetilde{\BP}_*(\wedge^{n-1}  \Z/p) \otimes_{\BP_*} N_*.
\end{equation} 
Because Theorem \ref{main} holds for $n-1$ we have a canonical isomorphism 
\[
    \widetilde{\BP}_*(\wedge^{n-1}  \Z/p) \otimes_{\BP_*} N_* = \bigoplus_{J_n = N_*} J_1 \otimes_{\BP_*} \otimes \cdots \otimes_{\BP_*} J_n
 \]
where the subscript $J_n = N_*$ means that we restrict ourselves to the tensor product summands 
with $J_n = N_*$.  Now consider the commutative diagram 
\[
   \xymatrix{
      \widetilde{\BP}_*(\wedge^{n-1}  \Z/p) \otimes_{\BP_*} N_* \ar[d]^{\cong}  \ar@{^{(}->}[r]  & \widetilde{\BP}_*(\wedge^{n}  \Z/p) \ar[d]^{{\rm proj.}\,  \circ \, \Gamma_n}\\ 
     \bigoplus_{J_n = N_*} J_1 \otimes  \cdots \otimes J_n  \ar[r] & \bigoplus_{J_1 \otimes \cdots \otimes J_n  
     \neq N_*^n, J_n  = N_*} J_1 \otimes \cdots \otimes J_n
 } 
\]
where the horizontal map on the bottom is a projection.


By assumption the right hand vertical map is surjective after 
restriction to $K_* \subset \widetilde{\BP}_*( \wedge^n \Z/p)$. Together with 
equation \eqref{summe} this implies 
\[
    \widetilde{\BP}_*(\wedge^n \Z/p) = K_* +  \im \Phi_n
\]
and with the assumption $K_* \cap \im \Phi_n= 0$ we conclude
\[
    \widetilde{\BP}_*(\wedge^n \Z/p) \cong K_* \oplus \im \Phi_n 
\]
as $\BP_*$-modules and $\BP_* \BP$-comodules.

The iterated K\"unneth map $\Phi_n$ is injective 
\cite[Corollary 3.3]{JW} (this also follows from the injectivity of the usual K\"unneth map and 
our assumption that Theorem \ref{main} holds for $n-1$). From this  
Theorem \ref{splitting} follows  for $n$. 

Once this has been established Theorem \ref{main} for $n$ follows by induction using the splitting 
$\Psi_n$ and the canonical isomorphisms $\Tor_{\BP_*}(N^k_*,N_*)_{*-1} \cong 
N_*^k \otimes_{\BP_*} L_k$ from Corollary \ref{wonderful} for $k \leq n-1$. 
Indeed, using 
Theorem \ref{main} for $n-1$, we then obtain the isomorphism 
\[
\big( \Tor_{\BP_*}( \widetilde{\BP}_*(\wedge^{n-1}  \Z/p), N_*)\big)_{*-1} \cong \big(\Tor_{\BP_*}(\bigoplus
     J_1 \otimes \cdots \otimes J_{n-1}, N_*)\big)_{*-1} \cong \bigoplus_{J_n = L_k} J_1 \otimes \cdots \otimes J_n.
\]
And this, together with the fact that the map 
\[
  \widetilde{\BP}_* ( \wedge^n \Z/p) \to \big(\Tor_{\BP_*}(\widetilde{\BP}_*(\wedge^{n-1} \Z/p), N_*)\big)_{*-1}
\]
in the K\"unneth exact sequence is induced by $ \id \wedge \pi :  \wedge^n B \Z/p \to  \wedge^{n-1} B\Z/p \wedge \CP^{\infty}$,
concludes the proof of Theorem \ref{main} for $n$. 
\end{proof}

For the proof of Theorem \ref{summary} observe that the homology $\HH_* ( C_* \otimes_{\BP_*} \cdots \otimes_{\BP_*} C_*)$ (we now write $C_*$ instead of $C_*^{BP}$) can be computed by induction in a completely analogous fashion
as $\widetilde{\BP}_*(\wedge^n \Z/p)$. The only adjustment consists in 
replacing the topological map $\pi :  B \Z/p \to \CP^{\infty}$
by the chain map 
\begin{eqnarray*} 
   \epsilon:   C_* & \to & F_1 \\
        c_{d}            & \mapsto & \begin{cases} y_m {\rm~for~} d = 2m, \\ 0 {\rm~for~} d {\rm~odd}, \end{cases} 
\end{eqnarray*} 
where $F_1$ is equipped with the zero differential. 
We hence get a K\"unneth sequence (setting $C_*^k := C_* \otimes_{\BP_*}  \cdots \otimes_{\BP_*}  C_*$)
\[
    0 \to \HH_*(C_*^{n-1} ) \otimes_{\BP_*} \HH_*(C_*) \to \HH_*(C_*^{n}) \to \big(\Tor_{\BP_*}(\HH_*
    (C_*^{n-1}) , \HH_*(C_*))\big)_{*-1} \to 0 
\]
induced by the map 
\[
     \HH_*(C_*^{n}) \to \HH_*(C_*^{n-1}) \otimes_{\BP_*} F_1 
\]
which on the chain level is given by $\id \otimes \epsilon : C_*^{n} \to C_*^{n-1} \otimes F_1$. 

This and an iterated use of   Corollary \ref{wonderful} show that there is a canonical isomorphism 
of $\BP_*$-modules
\[
 \HH_*(C_*^n) \cong \bigoplus J_1 \otimes_{\BP_*}  \cdots \otimes_{\BP_*} J_n                                                                                                      
\]
much as in Theorem \ref{main}. 

It remains to find a submodule $K_* \subset \widetilde{\BP_*}(\wedge^n \Z/p)$ enjoying the 
properties described in Proposition \ref{zerl}. 

\begin{defn} \label{inessential} Consider the  $\BP_*$-submodule of $\BP_*((\Z/p)^n)$which is  generated
by the images $\phi_*(\BP_*((\Z/p)^k))$, where  $\phi : (\Z/p)^k \to (\Z/p)^n$ is a group homomorpism and 
$k < n$. The image of this submodule in $\widetilde{\BP}_*(\wedge^n \Z/p)$ is called the 
{\em (reduced) inessential Brown-Peterson homology} of $(\Z/p)^n$. We denote this submodule by $K_*$. 
This is a $\BP_* \BP$-sub-comodule of $\widetilde{\BP}_*(\wedge^n \Z/p)$.
\end{defn}

Recall that the {\em essential cohomology} of a group $G$ is defined as the ideal 
in  $\HH^*(G)$ consisting of those classes that restrict to $0$ under all inclusions of proper subgroups 
$H < G$. Definition \ref{inessential} can be viewed as a dual notion for Brown-Peterson group homology. 

\begin{thm} \label{magic} The module $K_*$ 
fulfills the requirements of Proposition \ref{zerl}.
\end{thm} 

This will be proved in Section \ref{bewmagic} below. 


\section{Preliminaries on $N_*^k$} \label{prelim}

We need some preparation concerning the structure of $N_*^k =\widetilde{\BP}_*(\Z/p) \otimes_{\BP_*} \cdots
\otimes_{\BP_*} \widetilde{\BP}_*(\Z/p)$ for $k \geq 1$, extending 
 results in \cite[Section 3]{JW}, from which we borrow our notation (with the exception that 
 $z_m \in N_{2m+1}$ for $m \geq 0$).

For $I= (i_1, \ldots, i_{k})\in \N^k$ we write $z_I := z_{i_1} \otimes \cdots \otimes z_{i_{k}}\in N_*^k$. 
The $\BP_*$-module $N_*^k$ is generated by the elements $z_I$. It is worthwhile to consider the increasing $\BP_*$-module filtration of $N_*^{k}$ given by 
\[
    \mathcal{F}_J (N^k) :=  \langle z_I | I \leq J \rangle_{\BP_*} \subset N_*^{k}
\]
using the lexicographic order on $\N^k$. Let $E_{*}(N_*^k)$ be the associated graded $\BP_*$-module. 

We know from \cite[Theorem 3.2]{JW} 
 that $E_*(N_*^k)$ is a free module over $\BP_* / (v_0, \ldots, v_{k-1})$. This implies the following 
 non-squeezing result. 

\begin{lem} \label{squeeze} Let $k >  \ell$ and let 
\[
   \phi :    N_*^k \to N_*^{\ell} 
\]
be a not necessarily grading preserving $\BP_*$-linear map. Then $\phi$ is equal to zero. 
\end{lem} 

\begin{proof} We consider multiplication by powers of the generator $v_{\ell}  \in \BP_*$. 
On the one hand,  every element in 
$N_*^k$ is $v_{\ell}$-torsion because multiplication by $v_{\ell}$ is zero on $E_*(N_*^{k})$. 
On the other hand,  if $ c \neq 0\in N_*^k$, then for all $\nu \geq 0$ we have 
$(v_{\ell})^{\nu} \cdot c \neq 0$ because $E_*(N_*^{\ell})$ is free over $\BP_* / (v_0, \ldots, v_{\ell-1})$. 
From this our assertion follows. 
\end{proof}



Some  of our later computations will first be carried out in singular homology and then 
translated to $\BP$-homology. For this transition we need the following preliminaries. 

Let  $H_* := \widetilde{\HH}_*(\Z/p)$ and let $u_* : N_* \to H_*$ be the homological 
orientation. The module $H_*$ is generated by 
\[
   h_{m} := u(z_{m}) \in H_{2m+1},
\]
where $ m \geq 0$. As for $\BP$-homology we have a filtration 
\[
     \mathcal{F}_J (H_*^k) :=  \langle h_I | I \leq J \rangle \subset H_*^{k} := H_* \otimes \cdots 
     \otimes H_*
\]
and an associated graded $\F_p$-module $E_*(H_*^k)$, which is canonically isomorphic to $H_*^k$. 

In the following we fix $k$ and consider the ideal 
\[
   R := (v_k, v_{k+1}, \ldots ) \subset \BP_* . 
\]

\begin{lem} \label{kernel} The kernel of the (obviously surjective) map 
\[
   u_*^k := u_* \otimes \cdots \otimes u_* :     N_*^k \to H_*^k
\]
is equal to $R   \cdot N_*^k$. 
\end{lem}

\begin{proof} It is clear that $R \cdot N_*^k$ is contained in the kernel of $u_*^k$. 

For the reverse inclusion we observe that the  map $u_*^k$ induces a canonical map 
\[
    \omega : E_*(N_*^k) \to E_*(H_*^k) .
\]
Let $c \in \ker u_*^k$ 
and let $J$ be minimal with $c \in \mathcal{F}_J(N_*^k)$. Then the class $[c] \in (E_J)_*(N_*^k)$ 
lies in the kernel of $\omega$. Because $(E_J)_*(N_*^k)$ is free over $\BP_* / (v_0, \ldots, v_{k-1})$ 
with basis $([z_I])$ and $(E_J)_*(H_*^k)$ is free over $\F_p = \BP_*/(v_0, v_1, \ldots)$ with basis 
$( [h_I])$ 
we get
\[
   c \in R \cdot N_*^k + \mathcal{F}_{J-1}(N_*^k).
\]
As  $R \cdot N_*^k \subset \ker u_*^k$, the assertion follows 
by induction on the filtration degree $J$. 
\end{proof} 

\begin{prop} \label{surj} 

Assume that $I$ is an index  set and that 
\[
    \phi : \bigoplus_I N_*^k \to \bigoplus_I  N_*^k 
\]
is a (not necessarily grading preserving) $\BP_*$-linear map which induces a surjection
\[
           \bigoplus_I H_*^k  \to \bigoplus_I H_*^k 
\]
after dividing out $R \cdot  \bigoplus_I N_*^k$. 
Then $\phi$ itself is surjective. 
\end{prop} 

\begin{proof} For $ j \geq 0$ consider the decreasing filtration 
\[
     \mathcal{F}_{j} \left(  \bigoplus_I N_*^k \right)  =  \bigoplus_I \mathcal{F}_j (N_*^k) 
\]
where 
\[
   \mathcal{F}_j(N_*^k) := R^j \cdot N_*^k  \subset N_*^k
\]
and $R^j = R \cdot \ldots \cdot R$ with $j$ factors. Note that in each single degree of $ \oplus_I N_*^k$ this filtration is 
finite. The associated graded $\BP_*$-modules is denoted $E_*(\oplus_I N_*^k)$. 

By Lemma \ref{kernel} and our assumption  the map $\phi$ induces a surjective map 
$E_*(\oplus_I N_*^k) \to E_*(\oplus_I N_*^k)$. 
From this the assertion of the proposition follows 
by a (finite in each degree)  induction on filtration degrees.

\end{proof} 


\section{Calculations of inessential Brown-Peterson homology} \label{compute} 

In this section we perform some explicit calculations that will be employed in the proof of 
Theorem \ref{magic}.
Throughout we assume that Theorems \ref{splitting} and  \ref{main} have been proven for $n-1$.

Let $k,\ell \geq 1$ with $k+\ell = n$.  Furthermore let $1 \leq \delta_1  \leq \cdots \leq \delta_{\ell} \leq  k$. We choose integers
\begin{eqnarray*}
   0 & \leq & \lambda_{1,1}, \ldots, \lambda_{1, \delta_1} < p \\
   0&  \leq & \lambda_{2,1}, \ldots, \lambda_{2, \delta_2} < p \\
   & \vdots & \\ 
   0 & \leq & \lambda_{\ell,1}, \ldots, \lambda_{\ell, \delta_{\ell}} < p 
\end{eqnarray*} 
and collect them to a vector $\Lambda= (\lambda_{1,1}, \ldots, \lambda_{1, \delta_1}, 
\ldots, \lambda_{\ell,1}, \ldots, \lambda_{\ell, \delta_\ell})$ of length $\delta_1 + \cdots + \delta_{\ell}$. 
For $(x_1, \ldots, x_k) \in (\Z/p)^k$ and $1 \leq j \leq \ell$ we set
\[
    y_{j} := \sum_{i=1}^{\delta_j} \lambda_{j,i} \cdot x_i \in \Z/p
\]
and use this to define the group homomorphism 
$\phi_{\Lambda}: (\Z/p)^k \to (\Z/p)^{k+\ell}$ by the formula
\[
    \phi_{\Lambda}(x_1, \ldots, x_k) := (x_1, \ldots, x_{\delta_1}, {\bf y_{1}}, x_{\delta_1+1}, \ldots, 
    x_{\delta_2}, {\bf y_{2}}, \ldots, {\bf y_{\ell}}, x_{\delta_{\ell}+1}, \ldots, x_{k}) . 
\]
We print $y_j$ in boldface to make the definition of $\phi_{\Lambda}$ more transparent: We just 
plug in $y_1, \ldots y_{\ell}$ at positions $\omega_j := \delta_{j} + j$ for $j = 1, \ldots, \ell$. 

Note that we get an induced map  $ \wedge^k B\Z/p \to \wedge^{k+\ell} B \Z/p$. In the following we study this map in $\BP$-theory. At first we prove a vanishing result.  

For each $1 \leq i \leq k+\ell$ we choose a space $X_i$ equal to $B\Z/p$ or $\CP^{\infty}$. 
Applying the maps $\pi: B \Z/p \to \CP^{\infty}$ at the positions $i$ with $X_i = \CP^{\infty}$ we obtain a map 
\[
    \gamma_{(X_1, \ldots, X_n)}  : \bigwedge_n B\Z/p  \to \bigwedge_{i=1}^{k + \ell} X_i 
\]
as before Theorem \ref{main}. 

We now  assume that we have at least one $X_i = \CP^{\infty}$. 
Then the number of $X_i = B\Z/p$ is at most $n-1$, and we get an 
induced map \label{seite} 
\begin{equation} \label{defntheta}
  \Theta_{\Lambda}  :  N_*^k  \to \widetilde{\BP}_*( \wedge^k \Z/p) \stackrel{(\phi_{\Lambda})_*}{\longrightarrow}  \widetilde{\BP}_*(\wedge^{k+\ell}  \Z/p) \to  \bigotimes_{i=1}^{k+\ell} \widetilde{\BP}_*(X_i)
\end{equation} 
using  Theorem \ref{splitting}. 

Now choose $1 \leq \alpha \leq  k+\ell = n$ and define 

\begin{itemize} 
   \item $r(\alpha)$ as the minimal index $1 \leq j \leq \ell$ with 
$\omega_j (= \delta_j + j) \geq \alpha$, 
   \item $n(\alpha)$ as  $\ell - r(\alpha) +1$, 
  \item $m(\alpha)$ as the number of indices 
$i \geq \alpha$ with  $X_i = \CP^{\infty}$.
\end{itemize} 
In other words, $n(\alpha)$ is is the number of components $y_j$ 
that appear  at a position at least $\alpha$ in $\phi_{\Lambda}(x_1, \ldots, x_k)$, and 
$m(\alpha)$ is  the number of 
$\CP^{\infty}$ at a position at least $\alpha$ in the product $\wedge_{i=1}^n X_i$. 

Recall that the map $\pi : B \Z/p \to \CP^{\infty}$ induces the zero map in reduced $\BP$-homology (for degree
reasons).  The following 
is a generalization of this fact.

\begin{prop} \label{vanish} If there is some $1 \leq \alpha \leq k+\ell$ with $m(\alpha) > n(\alpha)$, then the map $\Theta_{\Lambda}$ is equal to zero. 
\end{prop} 

We have the following important consequence, where we equip the set  of $\ell$-tuples 
$1 \leq j_1 <  \cdots < j_{\ell} \leq n$ with the lexicographic order. 

\begin{cor} \label{conseqlex}  

Let the number of $X_i = \CP^{\infty}$ be equal to $\ell$ and  assume that 
\[
      (\omega_1, \ldots, \omega_{\ell}) < (j_1,\ldots,  j_{\ell})
\]
where $1 \leq j_1 < \cdots < j_{\ell} \leq n$ are those indices with $X_{j_i} = \CP^{\infty}$. 
Then the map $\Theta_{\Lambda}$ is equal to zero. 
\end{cor}

\begin{proof}[Proof of Proposition \ref{vanish}] We assume that $X_i = B \Z/p$ for all $i < \alpha$, which is no loss of generality, because the maps $B \Z/p \to X_i = \CP^{\infty}$ with $i < \alpha$ can 
be applied later. 

We have
\[
     \phi_{\Lambda} = \phi_2 \circ \phi_1
\]
where 
\begin{eqnarray*}
  \phi_1:   (\Z/p)^{k} & \to & (\Z/p)^{k+n(\alpha)} \\
     (x_1, \ldots, x_k) & \mapsto & (x_1, \ldots, x_{\delta_{r(\alpha)}}, {\bf y}_{r(\alpha)}, \ldots, {\bf y_{\ell}}, x_{\delta_{\ell}+1}, \ldots, x_k) 
 \end{eqnarray*} 
and
\begin{eqnarray*} 
    \phi_2 : (\Z/p)^{k+n(\alpha)} & \to & (\Z/p)^{k +\ell} \\
      (x_1, \ldots, x_{k+n(\alpha)} ) & \mapsto & (x_1, \ldots, x_{\delta_1}, {\bf y}_{1}, \ldots, {\bf y}_{r(\alpha)-1}, x_{\delta_{r(\alpha)-1}+1} , \ldots, x_{k+n(\alpha)}).
\end{eqnarray*}
In other words: At first we plug in only $y_{r(\alpha)}, \ldots, y_{\ell}$, and then, in a second step, 
the remaining $y_1, \ldots, y_{r(\alpha)-1}$. 

For $1 \leq i \leq k+n(\alpha)$ we define
\[
    Y_i := \begin{cases} B\Z/p {\rm~for~} i \leq  \alpha - r(\alpha), \\
                 X_{i + r(\alpha)-1}  {\rm~for~} i > \alpha - r(\alpha). \end{cases}
\]
We then have a commutative diagram 
\[
   \xymatrix{
      \bigwedge^{k+n(\alpha)} B\Z/p \ar[r]  \ar[d]^{B\phi_2} & \bigwedge_{i=1}^{k+n(\alpha)} Y_i  \ar[d]^{\chi}\\
      \bigwedge^{k+\ell} B\Z/p                \ar[r]^{\gamma_{(X_i)}}                            & \bigwedge_{i=1}^n  X_i.
 } 
\]
Here the map $\chi$ is induced by 
\[
   \phi_2|_{(\Z/p)^{\alpha - r(\alpha) } \times 1} : (\Z/p)^{\alpha - r(\alpha)} \to (\Z/p)^{\alpha -1} .
\]
The assumption that $X_i = B \Z/p$ for all $i < \alpha$ is used here. 

It is hence  enough to prove that the composition 
\[
    N_*^k \to \widetilde{\BP}_*(\wedge^k \Z/p) \stackrel{(\phi_1)_*}{\longrightarrow} 
                 \widetilde{\BP}_*(\wedge^{k+n(\alpha)} \Z/p) \longrightarrow \bigotimes_{i=1}^{k+n(\alpha)} 
                 \widetilde{\BP}_*(Y_i)
\]
is equal to zero. But this holds by the non-squeezing Lemma \ref{squeeze}, because the number of indices $i = 1, \ldots, 
k+n(\alpha)$ with $Y_i = B \Z/p$ is equal to $k + n(\alpha) - m(\alpha) < k$ and $\widetilde{\BP}_*(\CP^{\infty})$ is a free $\BP_*$-module.

\end{proof} 

In the following we consider the maps  $\Theta_{\Lambda}$  for the special case $k = n-1$, 
 $\ell = 1$ and $\delta_1 = k$. 
 This means we are given $0 \leq \lambda_1, \ldots, \lambda_{k}  < p $ and a group homomorphism
\begin{eqnarray*} 
 \phi =  \phi_{(\lambda_1, \ldots, \lambda_{k})} :  (\Z/p)^{k} & \to & (\Z/p)^{k+1}  \\
         (x_1, \ldots, x_{k}) & \mapsto & (x_1, \ldots, x_{k}, \lambda_1 x_1 + \cdots + \lambda_{k} x_{k}) . 
\end{eqnarray*} 
We obtain induced maps 
\begin{eqnarray*} 
    \phi_* : N_*^{k} \to \widetilde{\BP}_*( \wedge^{k} \Z/p) & \to &  \widetilde{\BP}_* (  \wedge^{k+1} \Z/p), \\
    \phi_* : H_*^{k} \to \widetilde{\HH}_* ( \wedge^{k}  \Z/p) & \to & \widetilde{\HH}_*( \wedge^{k+1} \Z/p), 
\end{eqnarray*} 
and by composition with the map
\[
   \gamma_{(B\Z/p, \ldots, B\Z/p, \CP^{\infty})} : \bigwedge_{k+1} B\Z/p \to \bigwedge_{k}B \Z/p \wedge \CP^{\infty}
\]
we obtain the maps
\begin{eqnarray*} 
   \Theta_{(\lambda_1, \ldots, \lambda_k)} :    N_{*}^{k} \stackrel{\phi_*}{\to}   \widetilde{\BP}_* ( \wedge^{k+1} \Z/p) \to 
    \widetilde{\BP}_*(\wedge^{k} \Z/p \wedge \CP^{\infty}) \stackrel{\Psi_{k} \otimes \id}{\longrightarrow}  N_*^{k} \otimes_{\BP_*} 
    \widetilde{\BP}_*(\CP^{\infty}) \\
  \Theta_{(\lambda_1, \ldots, \lambda_k)}:  H_*^k \stackrel{\phi_*}{\to}  \widetilde{\HH}_*( \wedge^{k+1} \Z/p)  \to 
    \widetilde{\HH}_*(\wedge^{k} \Z/p \wedge \CP^{\infty}) \stackrel{\Psi_{k} \otimes \id}{\longrightarrow}  H_*^{k} \otimes 
    \widetilde{\HH}_*(\CP^{\infty}).
\end{eqnarray*} 
Here we use the splittings of the iterated K\"unneth map for $\BP_*$-theory (see  Theorem \ref{splitting}) 
and ordinary homology (recall that $k = n-1 < n$).

The generators $\beta_{m} \in \widetilde{\BP}_{2m} ( \CP^{\infty})$ induce generators of  $\HH_{2m}(\CP^{\infty})$ 
which we denote by the same symbol. Similarly to $(L_k)_*$ let 
  $(M_k)_*$ be the free graded $\F_p$-module on generators ${y}_{m}$ of degree $2m$, $0 <  m < p^k$. In 
addition to the canonical projection $\widetilde{\BP}_*(\CP^{\infty})  \to  (L_k)_*$ 
we  obtain a canonical projection 
\[
        \widetilde{\HH}_* ( \CP^{\infty})  \to  (M_k)_* .
\]
by sending $\beta_{m} \mapsto y_{m}$ similar as before. 

We now consider the compositions 
\begin{eqnarray*} 
   \Theta_{(\lambda_1, \ldots, \lambda_k)}:   N_*^k \stackrel{\Theta_{(\lambda_1, \ldots, \lambda_k)}}{\longrightarrow}  N_*^{k} \otimes_{\BP_*} \widetilde{\BP}_*( \CP^{\infty}) \to N_*^{k} \otimes_{\BP_*} L_{k} \\
   \Theta_{(\lambda_1, \ldots, \lambda_k)}:   H_*^k \stackrel{\Theta_{(\lambda_1, \ldots, \lambda_k)}}{\longrightarrow} H_*^{k} \otimes \widetilde{\HH}_*( \CP^{\infty}) \to H_*^{k} \otimes M_{k} 
\end{eqnarray*} 
for various $0 \leq \lambda_1 , \ldots, \lambda_{k} < p$. 
The first map is graded $\BP_*$-linear and the second map is graded $\F_p$-linear. 
Both maps are compatible with the orientation $\BP \to \HH$. 

In the following we use the indexing set
\[
    \Lambda_k := \{ \Lambda \in \{0, \ldots, p-1\}^k ~ | ~ (\lambda_1, \ldots, \lambda_k) \neq (0, \ldots, 0) \}. 
\]
The next calculation is central for the  present paper. 

\begin{prop} \label{homcomp} The map 
\begin{eqnarray*} 
    \bigoplus_{\Lambda \in \Lambda_k} H_*^k &  \to & H_*^k \otimes  M_k \\
    (x_{\Lambda}) & \mapsto & \sum_{\Lambda \in \Lambda_k} \Theta_{\Lambda}(x_{\Lambda} ) 
\end{eqnarray*} 
is surjective. 
\end{prop} 

\begin{proof} We work in unreduced cohomology with $\F_p$-coefficients.  Identifying
\[
    \HH^*((\Z/p)^k ;\F_p) \cong  \F_p[t_1, \ldots, t_k] \otimes \Lambda(s_1, \ldots, s_k) \text{ and } 
   \HH^*(\CP^{\infty};\F_p) \cong \F_p[t]
\]
where $t_1, \ldots, t_k, t$ are indeterminates of degree $2$ and $s_1, \ldots, s_k$ indeterminates of degree $1$, 
the map induced in $\F_p$-cohomology  by 
\[
    B (\Z/p)^k \stackrel{\phi_{(\lambda_1, \ldots, \lambda_k)} }{\longrightarrow}  B(\Z/p)^{k+1}  
    \stackrel{\id \times \pi}{\longrightarrow} B(\Z/p)^k \times \CP^{\infty} 
\]
satisfies
\begin{equation} \label{firsteq} 
    (  t_1^{m_1} s_1 \cdot  \ldots \cdot  t_k^{m_k} s_k ) \cdot t^{\nu} \mapsto (t_1^{m_1} s_1 \cdot \ldots \cdot t_k^{m_k} s_k) \cdot (\lambda_1 t_1 + \ldots + \lambda_k t_k)^{\nu} .
\end{equation} 
The $(p^k \times p^k)$-Vandermonde-matrix  
 \[ 
    X :=   \left( \begin{array}{cccc} 1 & (\lambda_1 t_1 + \cdots + \lambda_k t_k) & \cdots & (\lambda_1 t_1 + \cdots + \lambda_k t_k)^{p^k-1} \end{array} \right)_{ 0 \leq \lambda_1, \ldots , \lambda_k < p}
 \]
 (where the subscript parametrizes the rows) with entries in $\F_p[t_1, \ldots, t_k]$ has  determinant 
 \[ 
     \prod_{(\lambda_1, \ldots, \lambda_k) < (\mu_1, \ldots, \mu_k)} ((\lambda_1 - \mu_1)t_1 + \cdots + (\lambda_k - \mu_k)t_k) \neq 0,
 \]
 where we use the lexicographic order in the index set. Hence the column vectors of $X$ are linearly independent over 
 $\F_p[t_1, \ldots, t_k]$. 
  
In view of formula \eqref{firsteq}  this means that the map 
\[
\bigoplus_{ 0 \leq \lambda_1, \ldots, \lambda_k < p}  \phi_{(\lambda_1, \ldots, \lambda_k)}^* :  \bigotimes_k \HH^{odd}(\Z/p; \F_p)  \otimes \HH^{0 \leq 2m < 2p^k} (\CP^{\infty};\F_p) \to 
\bigoplus_{0 \leq \lambda_1, \ldots, \lambda_k < p} \bigotimes_k \HH^{odd}(\Z/p; \F_p) 
\]
is injective.  Dualizing this statement over $\F_p$ and using the identification 
$H_* = \HH_{odd}(\Z/p;\F_p)$ we conclude that the map 
\begin{eqnarray*} 
    \bigoplus_{0 \leq \lambda_1, \ldots, \lambda_k < p} H_*^k &  \to & H_*^k \otimes {\rm span}_{\F_p} 
    \{  \beta_0, \beta_1, \ldots, 
    \beta_{p^k-1} \} \\
    (x_{\Lambda}) & \mapsto & \sum_{\Lambda \in \{0, \ldots, p-1\}^{k}} \Theta_{\Lambda}(x_{\Lambda}) 
\end{eqnarray*} 
is  surjective. This implies our claim, because the component $(\lambda_1, \ldots, \lambda_k) = 
(0, \ldots, 0)$ maps isomorphically to $H_*^k \otimes \beta_0 $. 
\end{proof}

Together with Proposition \ref{surj} (where $I$ is an index set with $p^k -1$ elements) this implies

\begin{prop} \label{core} The $\BP_*$-linear map 
\begin{eqnarray*} 
    \bigoplus_{\Lambda_k } N_*^k &  \to & N_*^k \otimes_{\BP_*}  L_k \\
    (x_{\Lambda}) & \mapsto & \sum_{\Lambda \in \Lambda_k} \Theta_{\Lambda}(x_{\Lambda}) 
\end{eqnarray*} 
is surjective. 
\end{prop} 

As a useful corollary we have (cf.  \cite[Theorem 4.1]{JW}) 

\begin{cor} \label{wonderful} The composition 
\[
   (\Tor_{\BP_*}(N_*^k, N_*))_{*-1} \subset N_*^k \otimes_{\BP_*} F_1 = N_*^k \otimes_{\BP_*} \widetilde{\BP}_*
   (\CP^{\infty}) \to N_*^k \otimes_{\BP_*} L_k .
\]
is a $\BP_*$-linear  isomorphism. 
\end{cor} 

\begin{proof} The algebraic Conner-Floyd conjecture implies that the left and right hand sides have the same cardinalities in each degree, see \cite[Proof of Theorem 4.1]{JW}.  
By construction, for each $\Lambda \in \Lambda_k$, the map $\Theta_{\Lambda}$ factors as 
\begin{align*}
 N_*^k \stackrel{\phi_*}{\longrightarrow}    \widetilde{\BP}_*(\wedge^{k+1} \Z/p) \to \big(\Tor_{\BP_*}(\widetilde{\BP}_*(\wedge^{k} \Z/p), N_*)\big)_{*-1} 
    \stackrel{\Tor_{\BP_*}(\Psi_{k}, \id)}{\longrightarrow} \\ 
 \to    (\Tor_{\BP_*}(N_*^k, N_*))_{*-1}  \subset N_*^k \otimes_{\BP_*} F_1 \to N_*^k \otimes_{\BP_*}  L_k .
\end{align*}
Hence, by Proposition \ref{core}, the module $(\Tor_{\BP_*}(N_*^k, N_*))_{*-1}$ has at least as many elements as 
$N_*^k \otimes_{\BP_*}  L_k$ (in each degree). This completes the proof. 
\end{proof} 

We remark that this proof of Corollary \ref{wonderful} is different from the proof of \cite[Theorem 4.1]{JW}. 

Now we  return to the situation at the beginning of this section for a fixed choice of   $ 1 \leq \delta_1 \leq 
\cdots \leq \delta_{\ell} \leq k$,  and set
\[
    X_i := \begin{cases} B \Z/p , {\rm~if~} i \neq \omega_j  {\rm~for~all~} 1 \leq j \leq \ell, \\
                                         \CP^{\infty}, {\rm~if~} i = \omega_j {\rm~for~some~} 1 \leq j \leq \ell.
                                         \end{cases}
\]
This means that the components $X_i =  \CP^{\infty}$ exactly match the $y_1, \ldots, y_{\ell}$ in 
$\phi_{\Lambda}(x_1, \ldots, x_k)$. 

For each choice of $\Lambda   \in \{0,\ldots, p-1\} ^{ \delta_1 + \cdots + \delta_{\ell}}$ 
we compose the map 
\[
   \Theta_{\Lambda} : N_*^k  \to   \bigotimes_{i=1}^{k+\ell} \widetilde{\BP}_*(X_i)
\]
with projections $\widetilde{\BP}_*(\CP^{\infty}) \to L_{\delta_j}$ to obtain a map 
\[
  \Theta_{\Lambda } : N_*^k \to N_*^{\delta_1} \otimes L_{\delta_1} \otimes N_*^{\delta_2 - \delta_1} \otimes
    L_{\delta_2} \otimes \cdots \otimes N_*^{\delta_l  - \delta_{l-1}} \otimes L_{\delta_l} \otimes N_*^{k - \delta_l}.
\]
In the following we use the indexing set 
\[
   \Lambda_{\delta_1, \ldots, \delta_{\ell}} := \{ \Lambda \in \{0, \ldots, p-1\}^{ \delta_1 + \cdots + \delta_{\ell}}~  |~  (\lambda_{j,1}, \ldots, \lambda_{j, \delta_j}) \neq (0, \ldots, 0) {\rm~for~} 1 \leq j \leq \ell \}.
\]   
   
\begin{prop} \label{homcomplarge} 
The $\BP_*$-linear map (with tensor products over $\BP_*$)
\begin{eqnarray*} 
    \bigoplus_{\Lambda \in \Lambda_{\delta_1, \ldots, \delta_{\ell}}} N_*^k & \to & 
        N_*^{\delta_1} \otimes L_{\delta_1} \otimes N_*^{\delta_2 - \delta_1} \otimes
    L_{\delta_2} \otimes \cdots \otimes N_*^{\delta_l  - \delta_{l-1}} \otimes L_{\delta_l} \otimes N_*^{k - \delta_l}\\
      (x_{\Lambda}) & \mapsto & \sum_{\Lambda \in  \Lambda_{\delta_1, \ldots, \delta_{\ell}}}
        \Theta_{\Lambda}(x_{\Lambda}) 
\end{eqnarray*} 
is surjective.
\end{prop} 

\begin{proof} The corresponding fact in homology follows from an application of  Proposition \ref{homcomp}
separately for each tensor factor $L_{\delta_j}$, $j = 1, \ldots, \ell$. 
Proposition \ref{surj} then implies the claim. 
\end{proof}


\section{Proof of Theorem \ref{magic}} \label{bewmagic} 

We still assume that Theorem \ref{main} 
holds for $n-1$. In the following proposition we apply the map $\Theta_{\Lambda}$ from formula \eqref{defntheta} on page \pageref{seite}
to different choices of $(X_1, \ldots, X_n)$ and collect the results into a direct sum. For the 
target of the following map we refer to the conventions of Theorem \ref{main} (with 
additional restrictions, indicated by a subscript under the direct sum sign). In addition, for each tensor product $J_1 \otimes_{\BP_*} \cdots \otimes_{\BP_*} J_n$ appearing in this theorem, let  
$\sharp_{\CP^{\infty}} (J_1 \otimes \cdots  \otimes J_n)$ be the number of factors equal to some $(L_{\gamma})_*$. 
(Note that the number $k$ in Theorem \ref{main} has a different meaning  than the number $k$ in the 
next Proposition).

\begin{prop} \label{surjalmost} Let  $k, \ell \geq 1$ with $k + \ell = n$ and let $1 \leq \kappa \leq  n-1$. Then the map 
\begin{eqnarray*}
   \bigoplus_{\Lambda \in \Lambda_{\delta_1, \ldots, \delta_{\ell}}}  N_*^k&  \to & 
   \bigoplus_{\sharp_{\CP^{\infty}}(J_1 \otimes \cdots \otimes J_n)= \kappa}  J_1 \otimes \cdots \otimes J_n \\
    (x_{\Lambda}) & \mapsto & \left(  \sum_{\Lambda \in \Lambda_{\delta_1, \ldots, \delta_{\ell}}} \Theta_{\Lambda}(x_{\Lambda}) \right)_{\sharp_{\CP^{\infty}}(J_1 \otimes \cdots \otimes J_n)= \kappa} 
\end{eqnarray*} 
is surjective for $\kappa = \ell$ and is equal to zero for $\kappa> \ell$. 
\end{prop}

\begin{proof} The case $\kappa = \ell$  follows by  induction on the set of indices $1 \leq j_1 < \cdots < j_{\ell} \leq n$ with $J_{i_i}$ equal to some $ (L_{\gamma})_*$, equipped 
with the lexicographic order, from Corollary \ref{conseqlex}  and Proposition \ref{homcomplarge}.
The case $\kappa > \ell$ follows from  the non-squeezing Lemma  \ref{squeeze} and 
the fact that $\widetilde{\BP}_*(\CP^{\infty})$ is a free $\BP_*$-module.
\end{proof} 

This proposition shows that the map 
\[
  K_* \subset   \widetilde{\BP}_*(\wedge^n \Z/p) \stackrel{\Gamma_n}{\longrightarrow} \bigoplus_{J_1 \otimes \cdots \otimes J_n \neq  N_*^n}  J_1 \otimes \cdots \otimes J_n 
\]
 is surjective.  Hence  $K_*$ satisfies the first property stated in Proposition \ref{zerl}.

It remains to show that $K_* \cap \im \Phi_n = 0$ where $\Phi_n$ is the iterated K\"unneth map. 
 Similarly as in \cite[Proof of Cor. 3.3]{JW} we can reduce this claim to the behavior 
of the toral element in $\BP_*((\Z/p)^n)$. In the next proposition recall that $z_0 \in \widetilde{\BP}_*(\Z/p)$ 
denotes the class $[L^1 \to B \Z/p]$. 

\begin{prop} \label{reduce} Assume that 
\[   
   K_* \cap   \Phi_n( \BP_* \cdot ( z_0 \otimes \cdots \otimes z_0) )  = 0 .
\]
Then $K_* \cap  \im \Phi_n = 0$. 
\end{prop} 

\begin{proof} 
Assume
\[
   \Phi_n( \sum c_I z_I ) \in K_* 
 \]
where each $c_I  \neq 0 \mod (v_0, \ldots, v_{n-1})$. 
Let  $J$ such that 
the degree of $z_J$ is maximal with nonzero $c_J$. Applying cap 
products with elements $t_i := 1 \otimes \cdots \otimes t \otimes \cdots \otimes 1 \in \BP^{2} ( ( \Z/p)^n)$ 
(unreduced $\BP$-cohomology) where $t =\pi^*(x) \in \BP^2(\Z/p)$ and the subscript $i$ refers to 
the $i$-th tensor factor, we see 
 \[
     \Phi_n( c_J \cdot (z_0 \otimes \cdots \otimes z_0)) \in   K_* .
 \]
 Here we observe that taking cap products with $t_i$ restricts to maps 
 \[
    K_* \to K_{*-2}
\]
by the naturality of the cap product. 

Hence, assuming that $K_* \cap  \Phi_n( \BP_* \cdot ( z_0 \otimes \cdots \otimes z_0) )= 0$, 
we conclude $\Phi_n( c_J \cdot (z_0 \otimes \cdots \otimes z_0)) = 0$, and by the injectivity
of the iterated K\"unneth map $\Phi_n$ \cite[Corollary 3.3]{JW} we conclude $c_J = 0 \mod
(v_0, \ldots, v_{n-1})$. This is a contradiction and hence an element $\sum c_I z_I$ as
above does not exist. 
 
 \end{proof} 
 
The proof that $ K_* \cap \Phi_n( \BP_* \cdot ( z_0 \otimes \cdots \otimes z_0) ) = 0$ 
is based on the following fact, which led to a solution of the Conner-Floyd conjecture. 

\begin{thm}\cite[Theorem 10.3]{RW} Let $p$ be an odd prime. For the canonical element $\iota_n \in 
\BP_n(K(\Z/p,n))$ the annihilator ideal is equal to $(v_0, \ldots, v_{n-1})$. 
\end{thm} 

This  result remains valid for $p=2$, see \cite[Appendix]{JW}. 

\begin{cor} \label{nutz} Let 
\[
     \mu_n : B(\Zp)^n = B\Z/p \times \cdots \times B \Z/p \to K(\Z/p,n) 
\]
be the canonical map induced by the ring structure on the Eilenberg-MacLane spectrum. 

Then the induced map 
\[
  ( \mu_n)_* :   \BP_*((\Z/p)^n) \to \BP_* ( K(\Z/p,n)) 
\]
is injective on $\Phi_n(\BP_* \cdot (z_0 \otimes \cdots \otimes z_0)) \subset   \BP_*((\Z/p)^n)$. 
\end{cor} 

Writing $\HH^*((\Z/p)^n; \F_p) = \F_p[t_1, \ldots, t_n] \otimes \Lambda(s_1, \ldots, s_n)$ 
note that $\mu_n$ represents the element $s_1 \cdot \ldots \cdot  s_n \in \HH^n((\Z/p)^n; \F_p)$. 

Corollary \ref{nutz} combined with the following proposition finishes  the proof that
\[
   K_* \cap  \Phi_n( \BP_* \cdot ( z_0 \otimes \cdots \otimes z_0) ) = 0.
\]

\begin{prop} \label{stretch} Let $k < n$ and let $\phi : (\Z/p)^k \to (\Z/p)^n$ be a group homomorphism. Then the 
induced map 
\[
    B(\Z/p)^k \stackrel{B\phi}{\longrightarrow}  B(\Z/p)^n \stackrel{\mu_n}{\longrightarrow} K(\Z/p,n) 
\]
is null homotopic. 
\end{prop} 

\begin{proof} The given map defines a  class  $c \in \HH^n((\Z/p)^k ; \F_p)$, and it is null homotopic, 
if and only if $c = 0$. We compute 
\[
  c = \phi^*(s_1)  \cup \cdots \cup \phi^*(s_n).
\]
Now
\[
 \HH^1((\Z/p)^k;\F_p) = \bigoplus_{i=1}^k \HH^0(\Z/p;\F_p) \otimes \cdots \otimes \HH^1(\Z/p;\F_p) \otimes \cdots \otimes \HH^0(\Z/p;\F_p) 
\]
where the tensor factor $\HH^1(\Z/p;\F_p)$ sits at position $i$. This and the fact that 
$s_i \cup s_i  = 0 \in \HH^2(\Z/p;\F_p)$ for all $i$ (for $p$ odd!) implies that for any 
$c_1, \ldots, c_n \in \HH^1((\Z/p)^k;\F_p)$ we have $c_1 \cup \cdots \cup c_n = 0$, because 
$n > k$.

This implies $c=0$ and the proof of the proposition is complete. 
\end{proof}  

The proof of Theorem \ref{level} is now rather easy. We already know from Proposition 
\ref{surjalmost} that $\widetilde{\BP}^{(k)}(\wedge^n \Z/p)$ for $1 \leq k \leq n$ span the whole of 
$\widetilde{\BP}_*(\wedge^n \Z/p)$. It remains to show that for fixed $k$ and for 
$0 \leq \kappa \leq n-1$ 
the composition 
\[
       \widetilde{\BP}^{(k)}(\wedge^n \Z/p)  \hookrightarrow \widetilde{\BP}_*(\wedge^n \Z/p) \stackrel{\Gamma_n}{\longrightarrow} 
    \bigoplus_{\sharp_{\CP^{\infty}}(J_1 \otimes \cdots \otimes J_n)= \kappa }  J_1 \otimes \cdots \otimes J_n
\]
is zero, if $\kappa \neq n - k$. For $\kappa > n-k$ this follows from Lemma \ref{squeeze}, compare the proof of Proposition \ref{surjalmost}. If $\kappa < n-k$, we 
argue as follows. Let  $X_i = B\Z/p$ or $X_i = \CP^{\infty}$ for all $1 \leq i \leq n$ with exactly $\kappa$ copies 
of $\CP^{\infty}$. Modulo a permutation of factors this happens for $1 \leq i \leq \kappa$. Now the composition 
\[
   \BP_*(( \Z/p)^k) \stackrel{\phi_*}{\longrightarrow} \BP_*((\Z/p)^n) 
   \to \widetilde{\BP}_* (\wedge^{\kappa} \CP^{\infty}) \otimes_{\BP_*} \widetilde{\BP}_*( \wedge^{n-\kappa} 
   \Z/p) 
\]
intersects the image of 
\[
     \widetilde{\BP_*} (\wedge^{\kappa} \CP^{\infty}) \otimes_{\BP_*} N_*^{n - \kappa} \to 
     \widetilde{\BP_*} (\wedge^{\kappa} \CP^{\infty}) \otimes_{\BP_*} \widetilde{\BP_*}( \wedge^{n-\kappa} 
   \Z/p) 
\]
only in the zero element as $k < n - \kappa$, by an argument similar as for Proposition \ref{stretch}. This implies the claim 
for $\kappa < n-k$, and therefore the proof of Theorem \ref{level} is complete. 

We conclude this section with an example showing that $K_* \cap  \im \Phi_n  \neq 0$ 
can occur for $p =2$.  Let  
\[
  \alpha :  L^3 \to B \Z/2 \to \wedge^3 B\Z/2
\]
be the composition of the classifying map $L^3 \to B\Z/2$ with the diagonal map 
$\Delta: B \Z/2 \to B (\Z/2)^3$ and the canonical projection $B (\Z/2)^3 \to 
\wedge^3 B \Z/2$. 

We have 
\[
   \Delta^*(s_1 \cup s_2 \cup s_3) = s^3 \neq 0  \in \HH^3(\Z/2;\F_2) = \F_2
\]
where we use $\HH^*(\Z/2; \F_2) = \F_2[s]$ with a polynomial generator $s \in \HH^1(\Z/2 ; \F_2) = \F_2$. 

Hence, considering the toral element 
\[
   \beta :  L^1 \times L^1 \times L^1 \to B(\Z/2)^3 \to \wedge^3 B\Z/2,
\]
both of the maps $\alpha$ and $\beta$ induce the canonical map $\Z \to \Z/2$  in the third integral homology. Because 
the forgetful map 
\[
  \widetilde{\BP}_3(\wedge^3 \Z/2) \to   \HH_3( \wedge^3 \Z/2; \Z) = \Z/2
\]
is an isomorphism, this implies that 
\[
    [\alpha] =  [\beta] \in \widetilde{BP}_3(\wedge^3 \Z/2) 
\] 
and the inessential class $[\alpha]$ is equal to the toral class $[\beta]$. 

A similar observation was made in  connection with the Gromov-Lawson-Rosenberg conjecture 
for elementary abelian $2$-groups in \cite{MJ}. 


\section{Gromov-Lawson-Rosenberg conjecture for elementary abelian groups} 
\label{pscm} 

The research leading to this paper was  inspired by \cite[Theorem 5.6]{BR1}, which 
claims that the image of the forgetful map (induced by the homological orientation 
$\BP \to \HH$)
\[
  h :   \BP_*((\Z/p)^n) \to \HH_*((\Z/p)^n)
\]
has the following description: Let $1 \leq k \leq n$ and consider the 
classifying map of a product of lens spaces
\[
     L^{2m_1 +1} \times \cdots \times L^{2m_k+1} \to B\Z/p \times \cdots \times B\Z/p . 
\]
The image of the fundamental class defines an element in $\HH_*((\Z/k)^k)$ that 
can be  mapped to $\HH_*((\Z/p)^n)$ by some group homomorphisms $\phi : (\Z/p)^k 
\to (\Z/p)^n$. It is claimed that the image of $h$  is (in positive degrees) additively generated by elements of this 
special kind.

The homological version of  our Theorem \ref{generated} implies that this is indeed the case. 

However,  the proof given in \cite{BR1} is incorrect. More precisely, at the top of page 204 in {\em loc.\,cit.} it is claimed that each element in  $\Tor_{\BP_*}(N_*^{r-1}, N_*)$, respectively
the image of such an element in $\Tor_{\Z} ( H_*^{r-1}, H_*)$,
can be realized as a sum of matrix Toda brackets
\[
 \left\langle z_{m_1} \otimes \cdots \otimes z_{m_{k-1}} \otimes (z_{m_k } ,z_{m_k-4}, \ldots) \otimes 
 z_{m_k+1} \otimes \cdots \otimes z_{m_{r-1}}, A, \left( \begin{array}{c} \vdots \\
 z_{2j-5} \\ z_{2j-1} \end{array} \right) \right\rangle. 
\]
This amounts to the assertion that the image of the map 
\[
    \Tor_{\BP_*}(N_*^{r-1}, N_*) \to \Tor_{\Z}(H_*^{r-1}, H_*)
\]
is generated by the images of the maps 
\[
    N_*^{k-1} \otimes_{\BP_*} \Tor_{\BP_*}(N_*, N_*) \otimes_{\BP_*} N_*^{r-k-1} \to
     \Tor_{\BP_*}(N_*^{r-1}, N_*) \to \Tor_{\Z}(H_*^{r-1}, H_*)
\]
where $1 \leq k \leq r-1$. 

This statement is wrong. To show  this we use the commutative 
diagram 
\[
   \xymatrix{
    \bigoplus_{1 \leq k \leq r-1}  N_*^{k-1} \otimes_{\BP_*} \Tor_{\BP_*}(N_*, N_*) \otimes_{\BP_*} N_*^{r-k-1} \ar[r] \ar[d] & 
     \Tor_{\BP_*}(N_*^{r-1}, N_*) \ar[d] \\
    \bigoplus_{1 \leq k \leq r-1}   H_*^{k-1} \otimes \Tor_{\Z}(H_*,H_*) \otimes H_*^{r-k-1} \ar[r] & \Tor_{\Z}(H_*^{r-1}, H_*) } .
\]
and recall that according to Corollary 
\ref{wonderful}, for each $1 \leq k \leq r-1$ 
the image of the composition of the right vertical with the upper horizontal 
 map is isomorphic to $H_*^{r-1} \otimes M_{1}$, whereas the image of 
 the right vertical map  is isomorphic to $H_*^{r-1} \otimes M_{r-1}$. Here we 
 recall that $(M_k)_*$ denotes the free graded $\F_p$-module with one generator in each even 
degree $2, \ldots, 2p^{k}-2$.

Now we observe that 
\[
    (r-1) \cdot \dim_{\F_p} (H_*^{r-1} \otimes M_{1}) < \dim_{\F_p} ( H_*^{r-1} \otimes M_{r-1})
\]
in large degrees, if $r -1\geq 2$. From this  we conclude that the lower horizontal map in the above diagram 
cannot contain the whole of $H_*^{r-1} \otimes M_{r-1}$. This shows that a simple reduction to the case $r=2$ as 
envisaged in \cite{BR1} is impossible. A similar problem occurs in the proof 
of the analogous result \cite[Theorem 4.2]{BR2}. 

Because the methods developed in our paper, in particular Theorem \ref{generated},  directly apply to bordism theory 
they can be used to verify the Gromov-Lawson-Rosenberg conjecture for elementary abelian 
groups of odd order right away, avoiding a further reduction to singular homology 
or connective $K$-homology as in \cite{BR1, BR2}, which requires additional arguments \cite{RS} based on 
manifolds with Baas-Sullivan singularities. 

We shall explain this proof. 
In the following we denote by $f : X \to B \pi_1(X)$ the classifying map of the universal 
cover of some path connected topological space $X$ .

\begin{defn} Let $M$ be a closed oriented manifold of dimension $d$ and let $p$ be a prime. 
The manifold $M$ is called  {\em $p$-atoral}, if 
\[
    f^*(c_1) \cup \cdots \cup f^*(c_d) = 0 \in \HH^d(M; \F_p) 
\]
for all one dimensional classes $c_1, \ldots, c_d \in \HH^1(B\pi_1(M) ; \F_p)$.
\end{defn} 

We can now prove the Gromov-Lawson-Rosenberg conjecture for 
atoral manifolds with elementary abelian 
fundamental groups of odd order, compare
\cite[Theorem 5.8]{BR1} and \cite[Theorem 2.3]{BR2}.

\begin{thm} \label{GLR} Assume that $M$ is a $p$-atoral manifold of dimension $d \geq 5$ 
with fundamental group $(\Z/p)^n$, where $p$ is an odd prime. Then the following assertions hold.

\begin{itemize} 
 \item If $M$ admits a spin structure, then $M$ admits a Riemannian metric of positive scalar curvature, if and only if $\alpha(M) = 0 \in {\rm KO}_d$, where $\alpha$ is the index invariant introduced 
 by Hitchin \cite{Hitchin} with values in the coefficients of  real $K$-homology. 
 \item If $M$ does not admit a spin structure, then $M$ admits a Riemannian 
metric of positive scalar curvature. 
\end{itemize} 
\end{thm} 

Note that $M$ in this theorem is automatically $p$-atoral, if $d > n$. This again uses the 
fact that for odd $p$ the one dimensional generator $s \in \HH^1(\Z/p; \F_p)$ squares to $0$. 

\begin{proof} First recall \cite{Hitchin} that $\alpha(M) = 0$, if $M$ is spin and can be 
equipped with a Riemannian metric of positive scalar curvature. Furthermore \cite{RS} 
a closed oriented smooth manifold $M$ of 
dimension $d \geq 5$ admits a Riemannian metric of positive scalar curvature, if and only if 
\begin{itemize} 
   \item $[f :  M \to B\pi_1(M)] \in \Omega^{Spin,+}_d( B \pi_1(M))$, in the case when $M$ admits a spin
            structure, 
   \item $[f : M \to B\pi_1(M)] \in \Omega^{SO,+}_d( B \pi_1(M) )$, in the case when the universal 
             cover $\widetilde M$ 
            does not admit a spin structure. 
\end{itemize} 
The superscript $+$ denotes bordism classes that can be represented by singular manifolds 
$X \to B\pi_1(M)$ (where $X$ is spin, respectively oriented) 
so that $X$ carries a Riemannian metric of positive scalar curvature. 

Now let $M$ be as in the theorem. Because $\pi_1(M) = (\Z/p)^n$ has odd order, the manifold $M$ admits 
a spin structure, if and only if its universal cover $\widetilde M$ does. 
Oriented and spin bordism are equivalent after 
localization at an odd prime $p$. This allows us to treat both cases in parallel to some extent, which 
we indicate by dropping the superscript $SO$ or $Spin$. 


Using Theorem \ref{generated} we can write
\[
    [f  : M \to B(\Z/p)^n] =  L_0  + \cdots +  L_n   \in \Omega_d((\Z/p)^n) 
\]
where each $L_k$ is the sum of bordism classes each of which is equal to 
the image of a class 
\[
  [  X  \times (L^{2m_1+1}\to B\Z/p) \times \cdots \times (L^{2m_k+1} \to B \Z/p) ] \in \Omega_d((\Z/p)^k)
\]
 under the map $\Omega_d((\Z/p)^k) \stackrel{\phi_*}{\longrightarrow} \Omega_d((\Z/p)^n)$ 
 induced by  some group homomorphism $\phi : (\Z/p)^k \to (\Z/p)^n$.   Here $L^{2m_i+1} \to B\Z/p$ are 
 classifying maps of standard lens spaces and $X$ is some conected closed manifold (spin or oriented, 
 respectively). 
In particular $L_0 \in \Omega_d \subset \Omega_d((\Z/p)^n)$. 
Without loss of generality we can assume that each group 
homomorphism $\phi$ is 
 injective.

For dimension reasons $L_k = 0$ for $k > d$. Because  $M$ is $p$-atoral, we also 
have $L_d =0$, because for 
$k = d$ each of the above summands is equal to $[(L^1 \to B\Z/p) \times \cdots \times (L^1 \to B\Z/p)]$, and 
the corresponding $\phi$ is injective. 
Now 
 \begin{itemize} 
    \item $\widetilde{\Omega}_{>1}^{Spin}(\Z/p) \subset \Omega_*^{Spin,+}(\Z/p)$, 
        see \cite[Theorem 1.3]{Ros}, and 
    \item  $\widetilde{\Omega}_{>1} ^{SO}(\Z/p) \subset \Omega_*^{SO,+}(\Z/p)$, 
        because the $\Omega_*^{SO}$-module 
             $\widetilde{\Omega}^{SO}_*(\Z/p)$ is 
              generated by lens spaces $[L^{2m+1} \to B \Z/p]$, $m \geq 0$, and $\Omega^{SO}_i = 
                 \Omega^{SO,+}_i$ for $i > 0$, see \cite{GL}. 
 \end{itemize} 
 This implies that each of the bordism classes  $L_1, \ldots, L_{d-1}$ lies in $\Omega^{+}_d((\Z/p)^n)$. 
 In particular,  in the spin case, we have $ \alpha(M) = \alpha(L_0)$. 
 
Hence our assertion holds, if  $L_0 \in \Omega^{+}_d$, where we assume $\alpha(L_0) = 0$ in the spin case. 
But this follows from  \cite{GL} (in the non-spin case) and 
\cite{Stolz} (in the spin case). 
\end{proof} 
 
The Gromov-Lawson-Rosenberg conjecture for toral manifolds with elementary abelian 
fundamental groups of odd order remains open. We also remark that Sven F\"uhring, in 
his Augsburg dissertation \cite[Corollary 5.2.2]{F}, gave a different proof of the second part of Theorem \ref{GLR}, 
using the notion of Riemannian metrics of positive scalar curvature on manifolds with Baas-Sullivan singularities.

 \end{document}